\documentclass[11pt,a4paper]{amsart}

\usepackage{amsmath, amsthm, amssymb}

\newcommand{\bC}{\mathbb{C}}
\newcommand{\bP}{\mathbb{P}}
\newcommand{\bN}{\mathbb{N}}
\newcommand{\rd}{\mathrm{d}}
\newcommand{\supp}{\operatorname{supp}}
\newcommand{\Int}{\operatorname{int}}
\newcommand{\bR}{\mathbb{R}}
\newcommand{\loc}{\operatorname{loc}}
\newcommand{\bD}{\mathbb{D}}
\newcommand{\bZ}{\mathbb{Z}}
\newcommand{\bQ}{\mathbb{Q}}

\numberwithin{equation}{section}
\theoremstyle{plain}
\newtheorem{theorem}{Theorem}[section]
\newtheorem{lemma}[theorem]{Lemma}

\newtheorem{mainth}{Theorem}

\theoremstyle{definition}

\newtheorem{notation}[theorem]{Notation}

\newtheorem*{acknowledgement}{Acknowledgement}

\theoremstyle{remark}
\newtheorem{remark}[theorem]{Remark}

\begin{document}  
  
\title[
The sequences of the derivatives of iterated polynomials]{Value distribution of
the sequences of the derivatives of iterated polynomials}

\author[Y\^usuke Okuyama]{Y\^usuke Okuyama}
\address{
Division of Mathematics,
Kyoto Institute of Technology,
Sakyo-ku, Kyoto 606-8585 Japan.}
\email{okuyama@kit.ac.jp}

\date{\today}

\begin{abstract}
We establish the equidistribution of the sequence of the averaged pullbacks
of a Dirac measure at any value in $\bC\setminus\{0\}$ under the derivatives
of the iterations of a polynomials $f\in\bC[z]$ of degree more than one towards
the $f$-equilibrium (or canonical) measure $\mu_f$ on $\bP^1$. 
We also show that for every $C^2$ test function on $\bP^1$,
the convergence is exponentially fast up to a polar subset of exceptional values 
in $\bC$. A parameter space analog of the latter quantitative result for
the monic and centered unicritical polynomials family is also established.
\end{abstract}


\maketitle

\section{Introduction}\label{sec:intro}

Let $f\in\bC[z]$ be a polynomial of degree $d>1$. Let
$\mu_f$ be the $f$-equilibrium (or canonical) measure on $\bP^1$, which 
coincides with
the harmonic measure $\mu_{K(f)}$ on the filled-in Julia set $K(f)$ of $f$
with respect to $\infty$. The exceptional set 
$E(f):=\{a\in\bP^1:\#\bigcup_{n\in\bN}f^{-n}(a)<\infty\}$
of $f$ contains $\infty$ and $\#E(f)\le 2$.
Brolin \cite[Theorem 16.1]{Brolin}
studied the value distribution of the sequence $(f^n:\bP^1\to\bP^1)$ of the iterations of $f$, and established
\begin{gather}
 \left\{a\in\bP^1:\lim_{n\to\infty}\frac{(f^n)^*\delta_a}{d^n}
 =\mu_f\text{ weakly on }\bP^1\right\}=\bP^1\setminus E(f),\label{eq:Brolin}
\end{gather}
which is more precise than the classical inclusion
$\partial K(f)\subset\overline{\bigcup_{n\in\bN}f^{-n}(a)}$
for every $a\in\bP^1\setminus E(f)$.
Here for every $h\in\bC(z)$ of degree $>0$ and
every Radon measure $\nu$ on $\bP^1$, 
the pullback $h^*\nu$ of $\nu$ under $h$ is 
a Radon measure on $\bP^1$ so that 
for every $a\in\bP^1$, when $\nu=\delta_a$,
$h^*\delta_a=\sum_{w\in h^{-1}(a)}(\deg_w h)\delta_a$ on $\bP^1$.
Pursuing the analogy between
the roles played by $E(f)$ in \eqref{eq:Brolin} and 
by the set of {\itshape Valiron exceptional values} in $\bP^1$ of
a transcendental meromorphic function on $\bC$,
Sodin \cite{Sodin92}, Russakovskii--Sodin \cite{RS95}, and 
Russakovskii--Shiffman \cite{RS97} (see also \cite{ES90}, \cite{okuyamaVNP})
studied the value distribution of
a {\itshape sequence of} rational maps between projective spaces
from the viewpoint of {\itshape Nevanlinna theory}, in a quantitative way
(cf.\ \cite[Chapter V, \S2]{Tsuji59}). 
Gauthier and Vigny \cite[1.\ in Theorem A]{GV16derivative}
studied the value distribution of
the sequence $((f^n)':\bP^1\to\bP^1)$ of the {\itshape derivatives} of iterations of 
a polynomial $f\in\bC[z]$ of degree $>1$ (cf.\ \cite{Yamanoi13}) 
possibly with a polar subset of exceptional values in $\bC\setminus\{0\}$,
in terms of dynamics of the {\itshape tangent map} $F(z,w):=(f(z),f'(z)w)$
on the tangent bundle $T\bC$. 
The aim of this article is to improve their result
in two ways.

The first improvement of \cite[1.\ in Theorem A]{GV16derivative} is
qualitative, but with no exceptional values. 

\begin{mainth}\label{th:derivative}
Let $f\in\bC[z]$ be of degree $d>1$. 
Then for every $a\in\bC\setminus\{0\}$, 
\begin{gather*}
 \lim_{n\to\infty}\frac{((f^n)')^*\delta_a}{d^n-1}=\mu_f
\end{gather*}
weakly on $\bP^1$.
\end{mainth}

In Theorem \ref{th:derivative}, the values $a=0,\infty$ are excluded since
it is clear that for every $n\in\bN$,
$((f^n)')^*\delta_\infty/(d^n-1)=\delta_\infty(\neq\mu_f)$, 
and it immediately follows from \eqref{eq:Brolin} and the chain rule that
$\lim_{n\to\infty}((f^n)'\delta_0)/(d^n-1)=\mu_f$ weakly
on $\bP^1$ if and only if $E(f)=\{\infty\}$.
In Gauthier--Vigny \cite[2.\ and 3.\ in Theorem A]{GV16derivative},
they also established a result similar to Theorem \ref{th:derivative} under the
assumption that $f$ has no Siegel disks (or the assumption that
$f$ is hyperbolic). Our proof of Theorem \ref{th:derivative}
is independent of their argument even in those cases.

The second improvement of \cite[1.\ in Theorem A]{GV16derivative}
is quantitative, 
but with an at most polar subset of exceptional values
in $\bC$.

\begin{mainth}\label{th:quantitative}
Let $f\in\bC[z]$ be of degree $d>1$, and suppose that $E(f)=\{\infty\}$. 
Then for every 
$\eta>\sup_{z\in\bC:\text{superattracting periodic point of }f}\limsup_{n\to\infty}(\deg_z(f^n))^{1/n}$, there is a polar subset $E=E_{f,\eta}$ in $\bC$ such that
for every $a\in\bC\setminus E$
and every $C^2$-test function $\phi$ on $\bP^1$,  
\begin{gather*}
\int_{\bP^1}\phi\rd\biggl(\frac{((f^n)')^*\delta_a}{d^n-1}-\mu_f\biggr)=o((\eta/d)^n)
\end{gather*}
as $n\to\infty$.
\end{mainth}

The proof of Theorem \ref{th:quantitative}
is based on Russakovskii--Shiffman \cite{RS97} mentioned above,
and on an improvement of it 
for the {\itshape sequence of the iterations} of a rational function
of degree $>1$
by Drasin and the author \cite{DOproximity}
(see also \cite{DS10} and \cite{Taflin11adv}
in higher dimensions).

\begin{remark}
 Under the assumption $E(f)=\{\infty\}$ in Theorem \ref{th:quantitative}, 
we have
 $\sup_{z\in\bC:\text{superattracting periodic point of }f}
 \limsup_{n\to\infty}(\deg_z(f^n))^{1/n}\in\{1,2,\ldots,d-1\}$, 
 and $=1$ if and only if 
 there is no superattracting cycles of $f$ in $\bC$.
Here we adopt the convention $\sup_{\emptyset}=1$. 
 In the case that $E(f)\neq\{\infty\}$, 
we point out the following better estimate than that in Theorem \ref{th:quantitative}
 \begin{gather*}
 \int_{\bP^1}\phi\rd\biggl(\frac{((f^n)')^*\delta_a}{d^n-1}-\mu_f\biggr)=O(nd^{-n})\quad\text{as }n\to\infty
 \end{gather*}
 for every $a\in\bC\setminus\{0\}$
 and every $C^2$-test function $\phi$ on $\bP^1$, with no exceptional values;
 indeed, we can assume that $f(z)=z^d$ without loss of generality (see Remark \ref{th:affine}), 
 and then $f^n(z)=z^{d^n}$ for every $n\in\bN$ and
 $\mu_f$ is the normalized 
Lebesgue measure 
 $m_{\partial\bD}$ on the unit circle $\partial\bD=\partial K(f)$.
 For every $a=re^{i\theta}$ ($r>0,\theta\in\bR$),
 every $C^1$-test function $\phi$ on $\bP^1$, and every $n\in\bN$, we have
 $\bigl|\int_{\bP^1}\phi\rd(((f^n)')^*\delta_a
 -\sum_{j=1}^{d^n-1}\delta_{e^{i(\theta+j\cdot 2\pi)/(d^n-1)}})/(d^n-1)\bigr|
 \le\|\phi\|_{C^1}\cdot\bigl|e^{(\log(rd^{-n}))/(d^n-1)}-1\bigr|
 \le\|\phi\|_{C^1}\cdot Cnd^{-n}$ for some $C>0$ independent of both $\phi$ and $n$,
 and if $\phi$ is $C^2$, 
 then by the {\itshape midpoint method} in numerically computing definite integrals,
 we also have $\bigl|\int_{\bP^1}\phi
 \rd\bigl(\sum_{j=1}^{d^n-1}\delta_{e^{i(\theta+j\cdot 2\pi)/(d^n-1)}}/(d^n-1)-m_{\partial\bD}\bigr)\bigr|\le\|\phi\|_{C^2}\cdot C'd^{-n}$ 
 for some $C'>0$ independent of both $\phi$ and $n$.
\end{remark}

Finally, let us focus on the (monic and centered) {\itshape unicritical polynomials family}
\begin{gather}
 f:\bC\times\bP^1\ni(\lambda,z)\mapsto z^d+\lambda=:f_\lambda(z)\in\bP^1\label{eq:unicritical} 
\end{gather}
of degree $d>1$. 
The parameter space analog of Theorem \ref{th:derivative} 
for the sequence $((f_\lambda^n)'(\lambda))$
in $\bC[\lambda]$ of the derivative of $f_\lambda^n$
at its unique critical value $z=\lambda$ in $\bC$ 
is also obtained by Gauthier--Vigny \cite[Theorem 3.7]{GV16derivative}.
We will also establish a parameter space analog of Theorem \ref{th:quantitative}.

\begin{mainth}\label{th:unicrit}
 Let $f$ be the monic and centered 
 unicritical polynomials family of degree $d>1$ defined as in \eqref{eq:unicritical}. 
 Then for every $\eta>1$, there is a polar subset $E=E_{f,\eta}$ in $\bC$ such that
 for every $a\in\bC\setminus E$ and every $C^2$-test function $\phi$ on $\bP^1$,
\begin{gather*}
 \int_{\bP^1}\phi(\lambda)\rd\left(\frac{((f_\lambda^n)'(\lambda))^*\delta_a}{d^n-1}-\mu_{C_d}\right)(\lambda)=O((\eta/d)^n)
\end{gather*}
 as $n\to\infty$. Here $C_d$ is the connectedness locus of the family $f$
 in the parameter space $\bC$ and $\mu_{C_d}$ is the harmonic measure on $C_d$
 with pole $\infty$.
\end{mainth}

The proof of Theorem \ref{th:unicrit}
is based on Russakovskii--Shiffman \cite{RS97} mentioned above,
and on a quantitative equidistribution of
superattracting parameters by Gauthier--Vigny \cite{GV15}.

In Section \ref{sec:background}, we recall a background from 
complex dynamics. 
In Sections \ref{sec:derivative}, \ref{sec:quantitative}, 
and \ref{sec:unicrit},
we show Theorems \ref{th:derivative}, \ref{th:quantitative}, 
and \ref{th:unicrit}, respectively. 

\begin{notation}\label{th:notation}
We adopt the convention $\bN=\bZ_{>0}$. For every $a\in\bC$ and every
$r>0$, set $\bD(a,r):=\{z\in\bC:|z-a|<r\}$.
Let $\delta_z$ be the Dirac measure on $\bP^1$ at each $z\in\bP^1$.
Let $[z,w]$ be the chordal metric on $\bP^1$
normalized as $[\cdot,\infty]=1/\sqrt{1+|\cdot|^2}$ on $\bP^1$
(following the notation in Nevanlinna's and Tsuji's books \cite{Nevan70,Tsuji59}).
Let $\omega$ be the Fubini-Study area element on $\bP^1$ normalized as
$\omega(\bP^1)=1$.
The Laplacian $\rd\rd^c$ on $\bP^1$ is normalized as
$\rd\rd^c(-\log[\cdot,\infty])=\omega-\delta_\infty$ on $\bP^1$.
\end{notation}

\section{Background}\label{sec:background}

\subsection{Dynamics of rational functions}
Let $f\in\bC(z)$ be of degree $d>1$. 
Let $C(f)$ be the critical set of $f$.
The Julia and Fatou sets of $f$ are defined by
$J(f):=\{z\in\bP^1:\text{the family }(f^n:\bP^1\to\bP^1)_{n\in\bN}\text{ is not normal at }z\}$ and $F(f):=\bP^1\setminus J(f)$, respectively.
A component of $F(f)$ is called a {\itshape Fatou component} of $f$. 
A Fatou component $U$ of $f$ is mapped by $f$ properly onto a Fatou component of $f$.
A Fatou component $U$ of $f$ is said to be {\itshape cyclic} 
if there is $n\in\bN$ such that $f^n(U)=U$. 
For more details on complex dynamics, see e.g.\ Milnor's book \cite{Milnor3rd}.

The $f$-{\itshape equilibrium $($or canonical$)$ measure} $\mu_f$
on $\bP^1$ is the unique probability Radon measure $\nu$ on $\bP^1$
such that 
\begin{gather}
 f^*\nu=d\cdot\nu\quad\text{on }\bP^1\label{eq:balanced} 
\end{gather}
and that $\nu(\{a\})=0$ for every $a\in E(f)$;
the {\itshape exceptional set} of $f$ is
$E(f):=\{a\in\bP^1:\#\bigcup_{n\in\bN}f^{-n}(a)<\infty\}
=\{a\in\bP^1:f^{-2}(a)=\{a\}\}$.
Then in fact $\supp\mu_f=J(f)$, and
for every $n\in\bN$, $\mu_{f^n}=\mu_f$ on $\bP^1$. 
For more details, see Brolin \cite{Brolin},
Lyubich \cite{Lyubich83}, Freire--Lopes--Ma\~n\'e \cite{FLM83}.

\subsection{Dynamics of polynomials}\label{sec:polynomial}
Let $f\in\bC[z]$ be of degree $d>1$. 
We note that $\infty\in E(f)$, $\#(C(f)\cap\bC)\le d-1$, and 
$C(f)\cap\bC=(\supp\rd\rd^c\log|f'|)\cap\bC$.

The filled-in Julia set $K(f)$ of $f$ is defined by
\begin{gather*}
 K(f):=\{z\in\bC:\limsup_{n\to\infty}|f^n(z)|<\infty\},
\end{gather*}
whose complement in $\bP^1$ coincides with the immediate superattractive basin
\begin{gather*}
 I_\infty(f):=\{z\in\bP^1:\lim_{n\to\infty}f^n(z)=\infty\}
\end{gather*}
of the superattracting fixed point $\infty$ of $f$; in particular,
$\lim_{n\to\infty}f^n=\infty$ locally uniformly on $I_\infty(f)$, and
$K(f)$ is a compact subset in $\bC$. We note that
$F(f)=I_\infty(f)\cup\Int K(f)$ and that $J(f)=\partial K(f)$.

By a standard telescope argument, there exists the locally uniform limit
\begin{gather*}
 g_f:=\lim_{n\to\infty}\frac{-\log[f^n(\cdot),\infty]}{d^n}
\end{gather*}
on $\bC$. Setting $g_f(\infty):=+\infty$, we have
$g_f\circ f=d\cdot g_f$ on $\bP^1$, and
for every $n\in\bN$, we also have $g_{f^n}=g_f$ on $\bP^1$.
The restriction of $g_f$ to $I_\infty(f)$ coincides with the 
Green function on $I_\infty(f)$ with pole $\infty$,
and the measure 
\begin{gather*}
 \mu_{K(f)}:=\rd\rd^c g_f+\delta_\infty\quad\text{on }\bP^1
\end{gather*}
coincides with the harmonic measure on $K(f)$ with pole $\infty$.
In particular, 
$\supp\mu_{K(f)}\subset\partial K(f)$, and in fact $\mu_{K(f)}=\mu_f$ on $\bP^1$.
The function $z\mapsto g_f(z)-\log|z|$ extends harmonically
to an open neighborhood of $\infty$ in $I_\infty(f)$
so the function $z\mapsto-\log[z,\infty]-g_f(z)$ 
extends continuously to $\bP^1$.

The following is substantially shown in Buff \cite[the proof of Theorem 4]{Buff03}.

\begin{theorem}[Buff]\label{th:deBranges}
Let $f\in\bC[z]$ be of degree $d>1$, and let $z_0\in\bC$.
If $g_f(z_0)\ge\max_{c\in C(f)\cap\bC}g_f(c)$,
then $|f'(z_0)|\le d^2\cdot e^{(d-1)g_f(z_0)}$, and
the equality {\itshape never} holds if 
$(C(f)\cap\bC)\cap I_\infty(f)\neq\emptyset$.
\end{theorem}

For more details on polynomial dynamics and
potential theory, see Brolin 
\cite[Chapter III]{Brolin}, and also Ransford's book \cite{Ransford95}.

\section{Proof of Theorem \ref{th:derivative}}
\label{sec:derivative}

Let $f\in\bC[z]$ be of degree $d>1$. 
For every $a\in\bC$ and every $n\in\bN$,
the functions $(\log|(f^n)'-a|)/(d^n-1)-g_f$ 
and $(\log\max\{1,|(f^n)'|\})/(d^n-1)-g_f$
extend continuously to $\bP^1$.
Set $a_d=a_d(f):=\lim_{n\to\infty}f(z)/z^d\in\bC\setminus\{0\}$.

\begin{remark}\label{th:affine}
Since the question is affine invariant, we could assume
$|a_d|=1$ without loss of generality, by replacing $f$ with $c^{-1}\circ f\circ c$
for such $c\in\bC\setminus\{0\}$ that $c^{d-1}=a_d^{-1}$ if necessary 
(for every $c\in\bC\setminus\{0\}$, $z\mapsto c\cdot z$ is also denoted by $c$).
In this article, we would not normalize $f$ as $|a_d|=1$ in order to
make it explicit which computations would be independent of such a normalization.
\end{remark}

\begin{lemma}\label{th:critical}
On $I_\infty(f)\setminus\bigcup_{n\in\bN\cup\{0\}}f^{-n}(C(f)\cap\bC)$,
\begin{gather*}
\lim_{n\to\infty}\left(\frac{\log|(f^n)'|}{d^n-1}-g_f\right)=0
\end{gather*}
locally uniformly.
\end{lemma}

\begin{proof}
For every $n\in\bN$ and every $z\in\bC$, by a direct calculation, we have
\begin{multline}
 \frac{\log|(f^n)'(z)|}{d^n-1}-\frac{\log\bigl|d^n\cdot a_d^{(d^n-1)/(d-1)}\bigr|}{d^n-1}\\
=\frac{1}{d^n-1}\int_{\bC}\log|z-u|(\rd\rd^c\log|(f^n)'|)(u)\\
=\frac{1}{d^n-1}\int_{\bC}\sum_{j=0}^{n-1}\left(\int_{\bC}\log|z-\cdot|\rd((f^j)^*\delta_w)\right)(\rd\rd^c\log|f'|)(w)\\
=\frac{1}{d^n-1}\int_{\bC}\sum_{j=0}^{n-1}\left(\log|f^j(z)-w|-\log|a_d|^{(d^j-1)/(d-1)}\right)(\rd\rd^c\log|f'|)(w)\\
=\frac{1}{d^n-1}\int_{\bC}\sum_{j=0}^{n-1}(\log[f^j(z),w]-\log[f^j(z),\infty]-\log[w,\infty])(\rd\rd^c\log|f'|)(w)\\
-\log|a_d|^{\frac{1}{d-1}-\frac{n}{d^n-1}}.\label{eq:equality}
\end{multline} 
Then noting that $g_f\circ f=d\cdot g_f$ on $\bP^1$, 
for every $n\in\bN$ and every $z\in\bP^1$, we have
\begin{multline}
\frac{\log|(f^n)'(z)|}{d^n-1}-g_f(z)\\
=\frac{1}{d^n-1}\int_{\bC}\Biggl(\sum_{j=0}^{n-1}\log[f^j(z),w]\Biggr)(\rd\rd^c\log|f'|)(w)\\
+\frac{d-1}{d^n-1}\sum_{j=0}^{n-1}\left(-\log[f^j(z),\infty]-g_f(f^j(z))\right)\\
+\left(-\int_{\bC}\log[w,\infty](\rd\rd^c\log|f'|)(w)
+\log d+\log|a_d|\right)\frac{n}{d^n-1},\label{eq:difference}
\end{multline}
which with 
$\sup_{z\in\bP^1}\left|-\log[z,\infty]-g_f(z)\right|<\infty$
completes the proof.
\end{proof}

\begin{lemma}\label{th:Fubini-Study}
There is $C=C_f>0$ such that 
for every $n\in\bN$ and every $z\in\bP^1$, 
\begin{gather}
\frac{\log\max\{1,|(f^n)'(z)|\}}{d^n-1}-g_f(z)\le\frac{Cn}{d^n-1}.\label{eq:upper}
\end{gather}
\end{lemma}

\begin{proof}
Set
\begin{multline}
 C=C_f:=(d-1)\cdot\sup_{z\in\bP^1}\left|-\log[z,\infty]-g_f(z)\right|\\
 +(d-1)\cdot\sup_{w\in C(f)\cap\bC}\left|\log[w,\infty]\right|
 +\log d+|\log|a_d||\in\bR_{>0}.\label{eq:nonzero}
\end{multline}
Then for every $n\in\bN$ and every $z\in\bC$,
from \eqref{eq:difference}, we have
$|(f^n)'(z)|\le e^{Cn}\cdot e^{(d^n-1)g_f(z)}$, which with $g_f\ge 0$ on $\bP^1$
completes the proof.
\end{proof}

We note that
$\max_{c\in\bigcup_{n\in\bN\cup\{0\}}f^{-n}(C(f)\cap\bC)}g_f(c)
=\max_{c\in C(f)\cap\bC}g_f(c)<\infty$ by $g_f\circ f=d\cdot g_f$ on $\bP^1$.

\begin{lemma}\label{th:potential}
For every $a\in\bC\setminus\{0\}$, 
\begin{gather*}
\lim_{n\to\infty}\int_{\bP^1}\left|\frac{\log|(f^n)'-a|}{d^n-1}
-g_f\right|\rd\omega=0.
\end{gather*}
\end{lemma}

\begin{proof}
Fix $a\in\bC\setminus\{0\}$. The sequence $((\log|(f^n)'-a|)/(d^n-1))$ 
 of subharmonic functions on $\bC$
 is locally uniformly bounded from above on $\bC$; indeed,
 by the chain rule and
$\liminf_{z\to\infty}|f'(z)|=+\infty$,
 for every $R>0$ so large that $\{|z|=R\}\subset I_\infty(f)\setminus\bigcup_{n\in\bN\cup\{0\}}f^{-n}(C(f)\cap\bC)$, we have
 $\liminf_{n\to\infty}\inf_{|z|=R}|(f^n)'(z)|=+\infty$, 
 which with the maximum modulus principle
 yields $\sup_{|z|\le R}|(f^n)'(z)-a|
 \le\sup_{|z|=R}2|(f^n)'(z)|$ for every $n\in\bN$ large enough.
 Then by Lemma \ref{th:Fubini-Study}, we have
$\limsup_{n\to\infty}\sup_{|z|\le R}(\log|(f^n)'-a|)/(d^n-1)
\le\sup_{|z|=R}g_f(z)<\infty$. 
By Lemma \ref{th:critical} and $g_f>0$ on $I_\infty(f)$,
for every compact subset $C$ in $I_\infty(f)\setminus\bigcup_{n\in\bN\cup\{0\}}f^{-n}(C(f)\cap\bC)$, we also have 
$1/2\le|((f^n)'-a)/(f^n)'|\le 2$ on $C$ for every $n\in\bN$ large enough,
so in particular
 \begin{gather}
 \lim_{n\to\infty}\left(\frac{\log|(f^n)'-a|}{d^n-1}-g_f\right)
 =\lim_{n\to\infty}\left(\frac{\log|(f^n)'|}{d^n-1}-g_f\right)=0\label{eq:uniform}
 \end{gather}
locally uniformly
on $I_\infty(f)\setminus\bigcup_{n\in\bN\cup\{0\}}f^{-n}(C(f)\cap\bC)$.

 Let $m_2$ be the Lebesgue measure on $\bC$.
 By a {\itshape compactness principle} for a locally uniformly
 upper bounded sequence of subharmonic functions on a domain in $\bR^m$
 which is not locally uniformly convergent to $-\infty$
 (see Azarin \cite[Theorem 1.1.1]{Azarin80}, 
 H\"olmander's book \cite[Theorem 4.1.9(a)]{Hormander83}),
 we can choose a sequence $(n_j)$ in $\bN$ tending to $+\infty$
 as $j\to\infty$ such that the $L^1_{\loc}(\bC,m_2)$-limit
 $\phi:=\lim_{j\to\infty}(\log|(f^{n_j})'-a|)/(d^{n_j}-1)$
exists and is subharmonic on $\bC$. 
Choosing a subsequence of $(n_j)$ if necessary,
we have $\phi=\lim_{j\to\infty}(\log|(f^{n_j})'-a|)/(d^{n_j}-1)$ 
Lebesgue a.e.\ on $\bC$. Then by \eqref{eq:uniform}, we have
$\phi\equiv g_f$ Lebesgue a.e.\ on 
$\bC\setminus(K(f)\cup\bigcup_{n\in\bN\cup\{0\}}f^{-n}(C(f)\cap\bC))$,
and in turn on $\bC\setminus K(f)$ by the subharmonicity of $\phi$
and the harmonicity of $g_f$ there.
Let us show that $\phi=g_f$ Lebesgue a.e.\ on the whole $\bC$, and then
$\lim_{n\to\infty}(\log|(f^n)'-a|)/(d^n-1)=g_f$ in $L^1_{\loc}(\bC,m_2)$,
which with the locally uniform convergence \eqref{eq:uniform} 
will complete the proof since
$\max_{c\in\bigcup_{n\in\bN\cup\{0\}}f^{-n}(C(f)\cap\bC)}g_f(c)
<\infty$ and the Radon-Nikodim derivative
$\rd\omega/\rd m_2$ is continuous so locally bounded on $\bC$.

By $\log(1/[w,\infty])-\log\max\{1,|w|\})\le\log\sqrt{2}$ on $\bC$
and Lemma \ref{th:Fubini-Study}, for every $n\in\bN$, we have
\begin{multline*}
\frac{\log|(f^n)'-a|}{d^n-1}-g_f\\
=\frac{\log[(f^n)',a]}{d^n-1}
+\left(\frac{\log(1/[(f^n)',\infty])}{d^n-1}-g_f\right)
+\frac{\log(1/[a,\infty])}{d^n-1}\\
\le\frac{C_f\cdot n}{d^n-1}+\frac{\log\sqrt{2}+\log(1/[a,\infty])}{d^n-1} 
\end{multline*}
on $\bC$, so $\phi\le g_f$ Lebesgue a.e.\ on $\bC$ and in turn on $\bC$
by the subharmonicity of $\phi$ and the continuity of $g_f$ on $\bC$.
Hence $\phi-g_f$ is $\le 0$ and is upper semicontinuous on $\bC$.

Now suppose to the contrary that the open subset
$\{z\in\bC:\phi(z)<g_f(z)\}$ in $\bC$ is non-empty.
Then by $\phi\equiv g_f$ on $\bC\setminus K(f)$,
there is a bounded Fatou component $U$ of $f$ containing a
component $W$ of $\{z\in\bC:\phi(z)<g_f(z)\}$.
Since $\phi\le g_f=0$ on $U\subset K(f)$, 
by the maximum principle for subharmonic functions,
we in fact have $U=W$. 

Taking a subsequence of $(n_j)$
if necessary, we can assume that $(f^{n_j}|U)$ is locally uniformly convergent
to a holomorphic function $g$ on $U$ as $j\to\infty$
without loss of generality. We claim that $g'\equiv a$ on $U$,
so we can say $g\in\bC[z]$;
indeed, fixing a domain $D\Subset U=W$,
by a version of Hartogs's lemma on subharmonic functions
(see H\"olmander's book \cite[Theorem 4.1.9(b)]{Hormander83}) and
the upper semicontinuity of $\phi$, we have
$\limsup_{n\to\infty}\sup_{\overline{D}}(\log|(f^{n_j})'-a|)/(d^{n_j}-1)
\le\sup_{\overline{D}}\phi<0$. Hence
$g'=(\lim_{j\to\infty}f^{n_j})'=\lim_{j\to\infty}(f^{n_j})'\equiv a$ on $D$, 
so $g'\equiv a$ on $U$ by the identity theorem for holomorphic functions. 

Hence, under the assumption that $a\neq 0$, 
the locally uniform limit $g$ on $U$ is non-constant. So by Hurwitz's theorem
and the classification of cyclic Fatou components, 
there is $N\in\bN$ such that
$V:=f^{n_N}(U)=g(U)(\supset g(\overline{D}))$
is a Siegel disk of $f$ and, setting $p:=\min\{n\in\bN:f^n(V)=V\}$,
that $p|(n_j-n_N)$ for every $j\ge N$. 
We can fix a holomorphic injection $h:V\to\bC$ such that 
for some $\alpha\in\bR\setminus\bQ$, setting $\lambda:=e^{2i\pi\alpha}$, we have
$h\circ f^p=\lambda\cdot h$ on $V$, so for every $j\ge N$,
$h\circ f^{n_j}=\lambda^{(n_j-n_N)/p}\cdot(h\circ f^{n_N})$
on $U$. Then taking a subsequence of $(n_j)$ 
if necessary, there also exists the limit 
\begin{gather*}
 \lambda_0:=\lim_{j\to\infty}\lambda^{(n_j-n_N)/p}
\end{gather*}
in $\partial\bD$,
so that $h\circ g=\lim_{j\to\infty} h\circ f^{n_j}=\lambda_0\cdot(h\circ f^{n_N})$
on $U$. In particular,
\begin{gather}
 h\circ f^{n_j}-h\circ g=(\lambda^{(n_j-n_N)/p}-\lambda_0)\cdot(h\circ f^{n_N})\label{eq:rotation}
\end{gather}
on $U$. Set $w_0:=h^{-1}(0)\in V$, so that $f^p(w_0)=w_0$, and
fix $z_0\in f^{-n_N}(w_0)\cap U$, so that $f^{n_j}(z_0)=w_0$ for every $j\ge N$
and $g(z_0)=\lim_{j\to\infty}f^{n_j}(z_0)=w_0$. 

We claim that
\begin{gather}
 \frac{\log|(f^{n_j})'(z_0)-a|}{d^{n_j}-1}
=\frac{\log|\lambda^{(n_j-n_N)/p}-\lambda_0|}{d^{n_j}-1}+O(d^{-n_j})\label{eq:compare} 
\end{gather}
as $j\to\infty$;
for, by the chain rule applied to both sides in \eqref{eq:rotation}
and $h'(w_0)\neq 0$ (and $g'(z_0)=a$), we have
\begin{gather*}
 (f^{n_j})'(z_0)-a=(\lambda^{(n_j-n_N)/p}-\lambda_0)\cdot(f^{n_N})'(z_0),
\tag{\ref{eq:rotation}$'$}
\end{gather*}
which also yields $(f^{n_N})'(z_0)\neq 0$ by
$(f^{n_j})'(z_0)=(f^{n_j-n_N})'(w_0)\cdot(f^{n_N})'(z_0)$
and the assumption $a\neq 0$. We also claim that 
\begin{gather}
 \liminf_{j\to\infty}\frac{1}{d^{n_j}}\log|\lambda^{(n_j-n_N)/p}-\lambda_0|\ge 0\label{eq:nonlinearity}
\end{gather}
(cf.\ \cite[Proof of Theorem 3]{OkuNonlinear}); indeed,
for every domain $D\Subset U\setminus f^{-n_N}(w_0)$, since 
$h^{-1}$ is Lipschitz continuous on
$h(\bigcup_{n\in\bN}(f^p)^n(f^{n_N}(D)))\cup g(D))\Subset h(V)$
and $\sup_D|h\circ f^{n_N}|>0$, from \eqref{eq:rotation}, we observe that
\begin{gather}
\frac{1}{d^{n_j}}\sup_{D}\log|f^{n_j}-g|
\le\frac{1}{d^{n_j}}\log|\lambda^{(n_j-n_N)/p}-\lambda_0|+O(d^{-n_j})
\tag{*}\label{eq:recurrence}
\end{gather}
as $j\to\infty$.
On the other hand, for every domain $\tilde{D}$ intersecting $\partial U$ in $\bC$, 
fixing $\tilde{z}\in\tilde{D}\cap I_\infty(f)\neq\emptyset$, we observe that
\begin{gather}
\liminf_{j\to\infty}\frac{1}{d^{n_j}}\sup_{\tilde{D}}\log|f^{n_j}-g|\ge g_f(\tilde{z})>0.
\tag{**}\label{eq:nonequi}
\end{gather}
Now fix $z_1\in U$ and $z'\in\partial U$ 
such that $\bD(z_1,|z'-z_1|)\subset U\setminus f^{-n_N}(w_0)$. Then
for every $\epsilon\in(0,|z'-z_1|)$, 
using Cauchy's estimate applied to $f^{n_j}-g\in\bC[z]$ around $z_1$, we have
\begin{multline*}
 |f^{n_j}-g|
\le\sum_{k=0}^{d^{n_j}}
\frac{\sup_{\partial\bD(z_1,|z'-z_1|-\epsilon)}|f^{n_j}-g|}{(|z'-z_1|-\epsilon)^k}|\cdot-z_1|^k\\
\le\biggl(\sup_{\bD(z_1,|z'-z_1|-\epsilon)}|f^{n_j}-g|\biggr)\cdot
\sum_{k=0}^{d^{n_j}}
\left(\frac{|z'-z_1|+\epsilon}{|z'-z_1|-\epsilon}\right)^k
\end{multline*}
on $\bD(z',\epsilon)$, so since $z'\in\bD(z',\epsilon)\cap\partial U$ and
$\bD(z_1,|z'-z_1|-\epsilon)\Subset U\setminus f^{-n_N}(w_0)$, by \eqref{eq:nonequi}
and \eqref{eq:recurrence}, we have
\begin{multline*}
0<\biggl(\liminf_{j\to\infty}\frac{1}{d^{n_j}}\log\sup_{\bD(z',\epsilon)}|f^{n_j}-g|\\
\le\liminf_{j\to\infty}\frac{1}{d^{n_j}}\log\sup_{\bD(z_1,|z'-z_1|-\epsilon)}|f^{n_j}-g|+\log\frac{|z'-z_1|+\epsilon}{|z'-z_1|-\epsilon}\\
\le\biggr)\liminf_{j\to\infty}\frac{1}{d^{n_j}}\log|\lambda^{(n_j-n_N)/p}-\lambda_0|
+\log\frac{|z'-z_1|+\epsilon}{|z'-z_1|-\epsilon}.
\end{multline*}
This yields \eqref{eq:nonlinearity} as $\epsilon\to 0$.

Once \eqref{eq:compare} and \eqref{eq:nonlinearity} are at our disposal,
using a version of Hartogs's lemma on subharmonic functions 
again,
we have
\begin{gather*}
\phi(z_0)\ge\limsup_{j\to\infty}\frac{\log|(f^{n_j})'(z_0)-a|}{d^{n_j}-1}
\ge\liminf_{j\to\infty}\frac{\log|\lambda^{(n_j-n_N)/p}-\lambda_0|}{d^{n_j}-1}\ge 0,
\end{gather*}
which contradicts $\phi<g_f=0$ on $U=W$.
\end{proof}

For every $a\in\bC\setminus\{0\}$ and every $C^2$-test function
$\phi$ on $\bP^1$, by Lemma \ref{th:potential}, we have
\begin{multline*}
 \left|\int_{\bP^1}\phi\rd\biggl(\frac{((f^n)')^*\delta_a}{d^n-1}-\mu_f\biggr)\right|
=\left|\int_{\bP^1}\phi\rd\rd^c\biggl(\frac{\log|(f^n)'(\cdot)-a|}{d^n-1}
-g_f\biggr)\right|\\
\le\left(\sup_{\bP^1}
\left|\frac{\rd\rd^c\phi}{\rd\omega}\right|\right)\cdot
\int_{\bP^1}\left|\frac{\log|(f^n)'(z)-a|}{d^n-1}
-g_f\right|\rd\omega(z)\to 0\quad\text{as } n\to\infty, 
\end{multline*}
where the Radon-Nikodim derivative $(\rd\rd^c\phi)/\rd\omega$ on $\bP^1$
is bounded on $\bP^1$. \qed

\section{Proof of Theorem \ref{th:quantitative}}
\label{sec:quantitative}

Let $f\in\bC[z]$ be of degree $d>1$, and suppose that $E(f)=\{\infty\}$.
Then
\begin{multline*}
 \sup_{z\in\bC:\text{superattracting periodic point of }f}\limsup_{n\to\infty}(\deg_z(f^n))^{1/n}\\
=\sup_{c\in C(f)\cap\bC:\text{periodic under }f}\limsup_{n\to\infty}(\deg_c(f^n))^{1/n}\in\{1,2,\ldots,d-1\}
\end{multline*}
(recall the convention $\sup_{\emptyset}=1$).
Set $a_d:=a_d(f)=\lim_{n\to\infty}f(z)/z^d\in\bC\setminus\{0\}$.
For every $n\in\bN$, the functions 
$(\log(1/[(f^n)',\infty])/(d^n-1)-g_f$ and
$(\log\max\{1,|(f^n)'|\})/(d^n-1)-g_f$
extend continuously to $\bP^1$.

\begin{lemma}\label{th:criticalquantitative}
For every 
$\eta>\sup_{c\in C(f)\cap\bC:\text{periodic under }f}\limsup_{n\to\infty}(\deg_c(f^n))^{1/n}$, 
\begin{gather*}
\int_{\bP^1}\left|\frac{\log(1/[(f^n)',\infty])}{d^n-1}-g_f\right|\rd\omega
=o((\eta/d)^n)
\end{gather*}
as $n\to\infty$.
\end{lemma}

\begin{proof}
For every $n\in\bN$, from \eqref{eq:difference},
we have
\begin{multline}
\int_{\bP^1}\left|\frac{\log|(f^n)'(z)|}{d^n-1}-g_f(z)\right|\rd\omega(z)\\
\le\frac{1}{d^n-1}\int_{\bC}\Biggl(\sum_{j=0}^{n-1}\int_{\bP^1}\log\frac{1}{[f^j(z),w]}\rd\omega(z)\Biggr)(\rd\rd^c\log|f'|)(w)+\frac{C_f\cdot n}{d^n-1},
\label{eq:integrated}
\end{multline}
where $C_f>0$ is defined in \eqref{eq:nonzero}.
By \cite[Theorem 2]{DOproximity}, 
for every $\eta>\sup_{c\in C(f)\cap\bC:\text{periodic under }f}\limsup_{n\to\infty}(\deg_c(f^n))^{1/n}$ and every 
$w\in\bC(=\bP^1\setminus E(f)$ under the assumption $E(f)=\{\infty\}$),
we have
\begin{gather*}
 \int_{\bP^1}\log\frac{1}{[f^n(z),w]}\rd\omega(z)=o(\eta^n)
\end{gather*}
as $n\to\infty$,
which with 
Lemma \ref{eq:upper} and
$0\le\log(1/[w,\infty])-\log\max\{1,|w|\}\le\log\sqrt{2}$ on $\bC$
completes the proof.
\end{proof}

\begin{lemma}\label{th:Sodin}
For every $\eta>1$, the Valiron exceptional set
\begin{gather*}
 E_V(((f^n)'),(\eta^n)):=\left\{a\in\bP^1:\limsup_{n\to\infty}
\frac{1}{\eta^n}\int_{\bP^1}\log\frac{1}{[(f^n)'(z),a]}\rd\omega(z)>0\right\}
\end{gather*}
of the sequence $((f^n)')$ of the derivatives of the iterations of $f$
with respect to the sequence $(\eta^n)$ in $\bR_{>0}$
is 
a polar subset in $\bP^1$.
\end{lemma}
\begin{proof}
 This is an application of Russakovskii--Shiffman 
 \cite[Proposition 6.2]{RS97} to the sequence $((f^n)')$ 
 in $\bC[z]$
 since $\sum_{n\in\bN}1/\eta^n<\infty$ for every $\eta>1$.
\end{proof}

For every 
$\eta>\sup_{c\in C(f)\cap\bC:\text{periodic under }f}\limsup_{n\to\infty}(\deg_c(f^n))^{1/n}$,
every $a\in\bC\setminus E_V(((f^n)'),(\eta^n))$, and every $C^2$-test function
$\phi$ on $\bP^1$, 
by Lemmas \ref{th:criticalquantitative} and \ref{th:Sodin},
we have
\begin{multline*}
 \left|\int_{\bP^1}\phi\rd\biggl(\frac{((f^n)')^*\delta_a}{d^n-1}-\mu_f\biggr)\right|\\
=\left|\int_{\bP^1}\phi\rd\rd^c\biggl(\frac{\log[(f^n)',a]}{d^n-1}
+\frac{\log(1/[(f^n)',\infty])}{d^n-1}-g_f\biggr)\right|\\
\le\left(\sup_{\bP^1}\left|\frac{\rd\rd^c\phi}{\rd\omega}\right|\right)\cdot\\
\left(\frac{1}{d^n-1}\int_{\bP^1}\log\frac{1}{[(f^n)'(z),a]}\rd\omega(z)
+\int_{\bP^1}\left|\frac{\log(1/[(f^n)'(z),\infty])}{d^n-1}-g_f\right|\rd\omega(z)\right)\\
=o((\eta/d)^n)\quad\text{as }n\to\infty,
\end{multline*}
where 
the Radon-Nikodim derivative $(\rd\rd^c\phi)/\rd\omega$ on $\bP^1$
is bounded on $\bP^1$. \qed

\section{Proof of Theorem \ref{th:unicrit}}\label{sec:unicrit}

Let $f:\bC\times\bP^1\ni(\lambda,z)\mapsto z^d+\lambda=:f_\lambda(z)\in\bP^1$
be the monic and centered 
unicritical polynomials family of degree $d>1$. For every $n\in\bN$,
$f_\lambda^n(\lambda),(f_\lambda^n)'(\lambda)\in\bC[\lambda]$
are of degree $d^n, d^n-1$, respectively.

\subsection{Background on the family $f$}
Recall the definitions in Subsection \ref{sec:polynomial}.
The following constructions are due to
Douady--Hubbard \cite{DH82} and Sibony.

For every $\lambda\in\bC$, 
$f_\lambda'(z)=d\cdot z^{d-1}$, so $C(f_\lambda)\cap\bC=\{0\}$ 
and $f_\lambda(0)=\lambda$.
The connectedness locus
$C_d:=\{\lambda\in\bC:\lambda\in K(f_\lambda)\}$
of the family $f$ is a compact subset in $\bC$, 
and $H_\infty=H_{d,\infty}:=\bP^1\setminus C_d$
is a simply connected domain containing $\infty$ in $\bP^1$. Moreover,
the locally uniform limit
\begin{gather*}
 g_{H_\infty}(\lambda):=g_{f_\lambda}(\lambda)=d\cdot g_{f_\lambda}(0)
=\lim_{n\to\infty}\frac{-\log[f_\lambda^n(\lambda),\infty]}{d^n}
\end{gather*}
exists on $\bC$.
Setting $g_{H_\infty}(\infty):=+\infty$,
the restriction of $g_{H_\infty}$ to $H_\infty$ coincides with the 
Green function on $H_\infty$ with pole $\infty$,
and the measure 
\begin{gather*}
 \mu_{C_d}:=\rd\rd^cg_{H_\infty}+\delta_\infty\quad\text{on }\bP^1
\end{gather*}
coincides with the harmonic measure on $C_d$
with pole $\infty$. In particular, $z\mapsto g_{H_\infty}(z)-\log|z|$ 
extends harmonically to an open neighborhood of $\infty$
in $H_\infty$, and $\supp\mu_{C_d}\subset\partial C_d$
(in fact, the equality holds). 

\subsection{Proof of Theorem \ref{th:unicrit}}
For every $n\in\bN$, 
$\lambda\mapsto(\log|(f_\lambda^n)'(\lambda)|)/(d^n-1)-g_{H_\infty}(\lambda)$ and
$\lambda\mapsto(\log\max\{1,|(f_\lambda^n)'(\lambda)|\})/(d^n-1)-g_{H_\infty}(\lambda)$ on $\bC$ extend continuously to $\bP^1$. 

\begin{lemma}\label{th:upperunicrit}
For every $n\in\bN$ and every $\lambda\in\bC$, 
 \begin{gather}
 \frac{\log\max\{1,|(f_\lambda^n)'(\lambda)|\}}{d^n-1}-g_{H_\infty}(\lambda)
 \le
\frac{n\log(d^2)}{d^n-1}.\tag{\ref{eq:upper}$'$}
 \end{gather}
\end{lemma}

\begin{proof}
For every $n\in\bN$ and every $\lambda\in\bC$, by $g_{f_\lambda^n}=g_{f_\lambda}$
on $\bP^1$
and $g_{f_\lambda}\circ f_\lambda=d\cdot g_{f_\lambda}$ on $\bP^1$, we have
$g_{f_\lambda^n}(\lambda)=g_{f_\lambda}(\lambda)=
d\cdot g_{f_\lambda}(0)\ge g_{f_\lambda}(0)
=\max_{c\in C(f_\lambda)\cap\bC}g_{f_\lambda}(c)
=\max_{c\in C(f_\lambda^n)\cap\bC}g_{f_\lambda^n}(c)$, 
so by Theorem \ref{th:deBranges}, we have
$|(f_\lambda^n)'(\lambda)|\le (d^n)^2e^{(d^n-1)g_{f_\lambda^n}(\lambda)}
=(d^n)^2e^{(d^n-1)g_{f_\lambda}(\lambda)}
=(d^n)^2e^{(d^n-1)g_{H_\infty}(\lambda)}$. This 
with $g_{H_\infty}(\lambda)\ge 0$ completes the proof.
\end{proof}

\begin{lemma}\label{th:FSparam}
\begin{gather*}
\int_{\bP^1}\left|\frac{\log(1/[(f_\lambda^n)'(\lambda),\infty])}{d^n-1}-g_{H_\infty}(\lambda)\right|\rd\omega(\lambda)
=O(n^2d^{-n})
\end{gather*} 
as $n\to\infty$.
\end{lemma}

\begin{proof}
 For every $n\in\bN$, by the third equality 
 in \eqref{eq:equality} for $f_\lambda$ evaluated at $z=\lambda$,
 we have
 \begin{gather*}
 \frac{\log|(f_\lambda^n)'(\lambda)|}{d^n-1}
 -\frac{n\log d}{d^n-1}
 =\frac{d-1}{d^n-1}\sum_{j=0}^{n-1}\log|f_\lambda^j(\lambda)|
 =\frac{d-1}{d^n-1}\sum_{j=0}^{n-1}\log|f_\lambda^{j+1}(0)|,
 \end{gather*} 
 so that 
 \begin{multline}
 \int_{\bP^1}\left|\frac{\log|(f_\lambda^n)'(\lambda)|}{d^n-1}-g_{H_\infty}(\lambda)\right|\rd\omega(\lambda)\\
 \le\frac{d-1}{d^n-1}\sum_{j=0}^{n-1}
 \int_{\bP^1}\left|\log|f_\lambda^{j+1}(0)|
 -d^j\cdot g_{H_\infty}(\lambda)\right|\rd\omega(\lambda)
 +\frac{n\log d}{d^n-1}\\
 =O(n^2d^{-n})\quad\text{as }n\to\infty
\tag{$\ref{eq:integrated}'$}
 \end{multline}
since by 
Gauthier--Vigny \cite[\S 4.3, Proof of Theorem A]{GV15}, we have
\begin{gather*}
 \int_{\bP^1}\left|\log|f_\lambda^{n+1}(0)|
 -d^n\cdot g_{H_\infty}(\lambda)\right|\rd\omega(\lambda)
 =O(n)
\end{gather*}
as $n\to\infty$.
This with Lemma \ref{th:upperunicrit} and
$0\le\log(1/[w,\infty])-\log\max\{1,|w|\}\le\log\sqrt{2}$ on $\bC$
completes the proof.
\end{proof}

\begin{lemma}\label{th:Sodinparam}
For every $\eta>1$, the Valiron exceptional set
\begin{gather*}
 E_V(((f_\lambda^n)'(\lambda)),(\eta^n))
:=\left\{a\in\bP^1:\limsup_{n\to\infty}
\frac{1}{\eta^n}\int_{\bP^1}\log\frac{1}{[(f_\lambda^n)'(\lambda),a]}\rd\omega(\lambda)>0\right\}
\end{gather*}
of the sequence $((f_\lambda^n)'(\lambda))$ in $\bC[\lambda]$
with respect to the sequence $(\eta^n)$ in $\bR_{>0}$ is 
a polar subset in $\bP^1$.
\end{lemma}
\begin{proof}
 This is an application of Russakovskii--Shiffman 
 \cite[Proposition 6.2]{RS97} to 
 the sequence $((f_\lambda^n)'(\lambda))$ in $\bC[\lambda]$
 since $\sum_{n\in\bN}1/\eta^n<\infty$ for every $\eta>1$.
\end{proof}

For every $\eta>1$,
every $a\in\bC\setminus E_V(((f_\lambda^n)'(\lambda)),(\eta^n))$, 
and every $C^2$-test function $\phi$ on $\bP^1$, 
by Lemmas \ref{th:FSparam} and \ref{th:Sodinparam}, 
we have
\begin{multline*}
 \left|\int_{\bP^1}\phi(\lambda)\rd\biggl(\frac{((f_\lambda^n)'(\lambda))^*\delta_a}{d^n-1}
-\mu_{C_d}\biggr)(\lambda)\right|\\
=\left|\int_{\bP^1}\phi(\lambda)\rd\rd^c\biggl(\frac{\log[(f_\lambda^n)'(\lambda),a]}{d^n-1}
+\frac{\log(1/[(f_\lambda^n)'(\lambda),\infty])}{d^n-1}-g_{H_\infty}(\lambda)\biggr)\right|\\
\le\left(\sup_{\bP^1}\left|\frac{\rd\rd^c\phi}{\rd\omega}\right|\right)\cdot\\
\biggl(\frac{1}{d^n-1}\int_{\bP^1}\log\frac{1}{[(f_\lambda^n)'(\lambda),a]}\rd\omega(\lambda)
+\int_{\bP^1}\left|\frac{\log(1/[(f_\lambda^n)(\lambda),\infty])}{d^n-1}
-g_{H_\infty}(\lambda)\right|\rd\omega(\lambda)\biggr)\\
=o((\eta/d)^n)\quad\text{as }n\to\infty,
\end{multline*}
where 
the Radon-Nikodim derivative $(\rd\rd^c\phi)/\rd\omega$ on $\bP^1$
is bounded on $\bP^1$. \qed

\begin{acknowledgement}
The author thanks the referee for a very careful scrutiny and
invaluable comments.
This research was partially supported by JSPS Grant-in-Aid 
for Scientific Research (C), 15K04924.
\end{acknowledgement}

\def\cprime{$'$}

\end{document}